\documentclass[11pt,letterpaper]{amsart}
\usepackage{graphicx}
\usepackage{amsmath} 
\usepackage{amsthm,amsfonts,amssymb,mathrsfs,amscd,amstext,amsbsy} 
\usepackage{epic,eepic} 
\usepackage{yfonts}
\usepackage{paralist,enumerate}
\usepackage[all]{xy}
\usepackage{hyperref}
\usepackage{mathtools}
\hypersetup{colorlinks}

\newtheorem{theorem}{Theorem}[section]

\newtheorem{theo}[theorem]{Theorem}
\newtheorem{lem}[theorem]{Lemma}

\newtheorem{que}[theorem]{Question}

\newtheorem*{Definition*}{Definition}
\def\qed{\hfill \ifhmode\unskip\nobreak\fi\quad\ifmmode\Box\else$\Box$\fi\\ }
\begin{document}
\title[Weights of oriented $S^1$-manifold with isolated fixed points]{Weights of circle actions on oriented manifolds with isolated fixed points}
\author{Donghoon Jang}
\thanks{This work was supported by a 2-Year Research Grant of Pusan National University.}
\address{Department of Mathematics and Institute of Mathematical Science, Pusan National University, Pusan, South Korea}
\email{donghoonjang@pusan.ac.kr}
\begin{abstract}
For an action of the circle group $S^1$ on a compact oriented manifold with isolated fixed points, there is a claim that weights at the fixed points occur in pairs. This phenomenon holds for other types of $S^1$-manifolds, e.g., (almost) complex, symplectic, and unitary  manifolds. A known proof of this claim assumes that the isotropy submanifolds are orientable. However, this assumption does not hold in general.
We show the claim without relying on that assumption.
\end{abstract}
\maketitle
\section{Introduction}

For an action of a Lie group on a manifold, the data on the fixed point set is essential information for classifying such a manifold. In certain circumstances, this local data uniquely determines the manifold.

For an action of the circle group $S^1$ with isolated fixed points on certain types of compact manifolds, such as (almost) complex, symplectic, or unitary, there is an interesting phenomenon: weights at the fixed points occur in pairs. 
For example, for an $S^1$-action on a compact almost complex manifold with a discrete fixed point set, for each integer $w$,
\begin{center}
$\displaystyle \sum_{p \in M^{S^{1}}} N_{p}(w)=\sum_{p \in M^{S^{1}}} N_{p}(-w)$,
\end{center}
where $N_{p}(w)$ is the number of times $w$ occurs as a weight at $p$ and the sum runs over all fixed points \cite{H}. As a consequence, weights $w$ and $-w$ occur in a pair. 
Moreover, in the pair $(w,-w)$, the weights $w$ and $-w$ can be chosen from distinct fixed points, which lie in the same component of the $\mathbb{Z}_w$-fixed point set $M^{\mathbb{Z}_w}$ of $M$ \cite{JT}; here, the group $\mathbb{Z}_w$ acts on $M$ as a cyclic subgroup of the circle group $S^1$. 
Also, see \cite{J4} for an extension of this result to unitary $S^1$-manifolds.

Now, suppose that $M$ is a compact, oriented manifold, and let the circle group $S^1$ act effectively on $M$ with a discrete fixed point set. In this case, unlike the case of almost complex manifolds, the sign of each weight at a fixed point is not well-defined, and we can choose each weight to be positive; for details, see the next section. Let $w \geq 1$ be an integer. It is known that if $w$ is odd, then the $\mathbb{Z}_w$-fixed point set $M^{\mathbb{Z}_w}$ is orientable (see Lemma~\ref{odd-ori}.)
As for other types of manifolds,
it has been assumed that when $w$ is even, a (connected) component of the $\mathbb{Z}_w$-fixed point set $M^{\mathbb{Z}_w}$ is also orientable, provided it contains an $S^1$-fixed point. This assumption led to a result, analogous to the case for other types of manifolds, stating that weights at the fixed points can be split into pairs such that, in each pair $(w_{p,i},w_{q.j})$, the weights $w_{p,i}$ and $w_{q,j}$ are equal and come from distinct fixed points $p \neq q$. Moreover, the fixed points $p$ and $q$ lie in the same component $F$ of $M^{\mathbb{Z}_w}$ of the $\mathbb{Z}_w$-fixed set of $M$, where $w=w_{p,i}=w_{q,j}$ and $F$ is orientable. 

On the other hand, it has been shown that when $w$ is even, a $\mathbb{Z}_w$-fixed point set $M^{\mathbb{Z}_w}$ containing an $S^1$-fixed point need not be orientable, provided $\dim M \geq 6$ \cite{W}. This non-orientability of an isotropy submanifold in an oriented $S^1$-manifold contrasts with the case of other types of manifolds; if an action of a group on a manifold preserves an (almost) complex, symplectic, or unitary structure, then its isotropy submanifold is also (almost) complex, symplectic, or unitary, respectively.

In this note, we show that weights can still be split into pairs, even when $M^{\mathbb{Z}_w}$ is not necessarily orientable for even $w$.

\begin{theo} [Theorem~\ref{match}$+$Theorem~\ref{odd}]
Let the circle group $S^1$ act on a compact oriented manifold $M$ with a discrete fixed point set. Then the multiset $\{w_{p,i}\}_{1 \leq i \leq n, p \in M^{S^1}}$ of weights $w_{p,i}$ at the fixed points $p$ can be split into pairs $(w_{p,i},w_{q,j})$ such that 
\begin{enumerate}
\item $w_{p,i}=w_{q,j}$, and 
\item $p \neq q$ or $i \neq j$ if $w_{p,i}$ is even, and $p \neq q$ if $w_{p,i}$ is odd. 
\end{enumerate}
\end{theo}

Thus, as in the case of other types of manifolds, it is natural to ask whether even weights can also be split into pairs such that, in each pair, the weights come from distinct fixed points.

\begin{que} \label{que}
Let the circle group act on a compact oriented manifold with a discrete fixed point set. Can the even weights at all fixed points be split into pairs $(w_{p,i},w_{q,j})$ such that $w_{p,i}=w_{q,j}$ and $p \neq q$?
\end{que}

For an odd weight $w_{p,i}$, we can require that the fixed points $p$ and $q$ lie in the same component of the $\mathbb{Z}_w$-fixed set $M^{\mathbb{Z}_w}$, which is an orientable submanifold of $M$ (see Theorem~\ref{odd}). Here, $w=w_{p,i}=w_{q,j}$.
Note that if $\dim M=4$, then any component of $M^{\mathbb{Z}_w}$ containing an $S^1$-fixed point is always orientable for all $w$; see \cite{JM, J5}. As a result, Question~\ref{que} has an affirmative answer when the dimension of the manifold is at most 4.

\section*{Acknowledgements}

The author would like to thank the anonymous referee for valuable comments and pointing out an error in the proof of Theorem~\ref{match} in an earlier version of this manuscript.

\section{Properties of weights}

For an action of a group $G$ on a manifold $M$, let $M^G$ denote its fixed point set; if $H$ is a subgroup of $G$, then the group $H$ also acts on $M$, and its fixed point set is denoted by $M^H$.

Let the circle group $S^1$ act on an oriented manifold $M$. Let $p$ be an isolated fixed point. The orientation on $M$ induces an orientation on the tangent space $T_pM$ at $p$. The tangent space $T_pM$ of $M$ at $p$ decomposes into the sum of real 2-dimensional irreducible $S^1$-equivariant vector spaces $L_{p,1},\cdots,L_{p,n}$, where $\dim M=2n$. Each vector space $L_{p,i}$ is isomorphic to a complex one-dimensional $S^1$-equivariant complex space, on which the circle group acts as multiplication by $g^{w_{p,i}}$ for all $g \in S^1 \subset \mathbb{C}$, for some non-zero integer $w_{p,i}$. For each $L_{p,i}$, we can choose an orientation of $L_{p,i}$, so that $w_{p,i}$ is positive; this gives an orientation on $L_{p,i}$ and hence on $T_pM=L_{p,1} \oplus \cdots \oplus L_{p,n}$. We call the positive integers $w_{p,1},\cdots,w_{p,n}$ the \textbf{weights} at $p$. The tangent space $T_pM$ has two orientations, one induced by the orientation on $M$ and the other induced by the orientation on $L_{p,1} \oplus \cdots \oplus L_{p,n}$. We define the \textbf{sign} of $p$, denoted $\epsilon_M(p)$, to be $+1$ if the two orientations agree and $-1$ otherwise. If there is no confusion, we shall write $\epsilon_M(p)$ as $\epsilon(p)$, omitting $M$.

For a circle action on a compact oriented manifold with isolated fixed points, the Atiyah-Singer index theorem yields the following formula.

\begin{theo} \emph{[\textbf{Atiyah-Singer index theorem}]} \cite{AS} \label{t22} Let the circle act on a $2n$-dimensional compact oriented manifold $M$ with a discrete fixed point set. Then the signature of $M$ satisfies
\begin{equation} \label{equation}
\displaystyle\textrm{sign}(M) = \sum_{p \in M^{S^1}} \epsilon(p) \cdot \prod_{i=1}^{n} \frac{1+t^{w_{pi}}}{1-t^{w_{pi}}}
\end{equation}
for all indeterminates $t$, and is a constant. \end{theo}

Let the circle group act on a compact oriented manifold $M$ with a discrete fixed point set. We shall use the geometric power series expansion $\frac{1}{1-x}=\sum_{i=0}^\infty x_i=1+x+x^2+\cdots$ to rewrite the Atiyah-Singer index formula \eqref{equation} as follows.
\begin{center}
$\displaystyle \textrm{sign}(M) = \sum_{p \in M^{S^1}} \epsilon(p) \prod_{i=1}^{n} \frac{1+t^{w_{p,i}}}{1-t^{w_{p,i}}} =\sum_{p \in M^{S^1}} \epsilon(p) \prod_{i=1}^{n} [(1+t^{w_{p,i}})\sum_{j=0}^{\infty} t^{jw_{p,i}}]=\sum_{p \in M^{S^1}} \epsilon(p) \prod_{i=1}^{n} (1+2\sum_{j=1}^{\infty} t^{jw_{p,i}})$
\end{center}
That is, the following equation holds.
\begin{equation} \label{eq2}
\displaystyle \textrm{sign}(M) = \sum_{p \in M^{S^1}} \epsilon(p) \prod_{i=1}^{n} (1+2\sum_{j=1}^{\infty} t^{jw_{p,i}})
\end{equation}
In Equation~\eqref{eq2}, a fixed point $p$ contributes the term 
\begin{center}
$\displaystyle \epsilon(p) \prod_{i=1}^{n} (1+2\sum_{j=1}^{\infty} t^{jw_{p,i}})$.
\end{center}
In this term, we compute the coefficient of the $t^w$-term.
It has a $t^w$-term if some non-negative integer combination of the weights of $p$ is equal to $w$. Suppose that some non-negative integer combination $\sum_{i=1}^n j_i w_{p,i}$ of the weights $w_{p,i}$ at $p$ is equal to $w$, where $j_i \geq 0$ for all $i$. For a multiset $\{j_1,\cdots,j_n\}$ of non-negative integers, let $m(j_1,\cdots,j_n)$ denote the number of positive elements in $\{j_1,\cdots,j_n\}$.
From this combination $\sum_{i=1}^n j_i w_{p,i}$, the coefficient of the term $\displaystyle t^{\sum_{i=1}^n j_i w_{p,i}}$ is $\epsilon(p) \cdot 2^{m(j_1,\cdots,j_n)}$. Therefore, summing over all non-negative integer combinations of the weights of $p$ that are equal to $w$, it follows that the coefficient of the $t^w$-term in $\epsilon(p) \prod_{i=1}^{n} (1+2\sum_{j=1}^{\infty} t^{jw_{p,i}})$ is
\begin{center}
$\displaystyle \epsilon(p) \sum_{j_1w_{p,1}+\cdots+j_nw_{p,n}=w , 0\leq  j_i} 2^{m(j_1,\cdots,j_n)}$.
\end{center}

We illustrate this with an example. Suppose that a fixed point $p$ has sign $+1$ and weights $\{1,2,3\}$ and $w=3$. The coefficient of $t^3$ in the expression
\begin{center}
$\displaystyle \epsilon(p) (1+2\sum_{j=1}^{\infty} t^{jw_{p,1}})(1+2\sum_{j=1}^{\infty} t^{jw_{p,2}})(1+2\sum_{j=1}^{\infty} t^{jw_{p,3}})$

$\displaystyle =(1+2\sum_{j=1}^{\infty} t^{j})(1+2\sum_{j=1}^{\infty} t^{2j})(1+2\sum_{j=1}^{\infty} t^{3j})$
 \end{center}
is $6=4+2$. Here, $4$ comes from $4t^3=2t \cdot 2t^2 \cdot 1$, where $2t$ comes from the first bracket, $2t^2$ comes from the second bracket, and $1$ comes from the third bracket ($j_1=1$, $j_2=1$, and $j_3=0$), and $2$ comes from $1 \cdot 1 \cdot 2t^3$, where the two $1$'s come from the first and second brackets and $2t^3$ comes from the third bracket ($j_1=0$, $j_2=0$, and $j_3=1$).

With this understood, we show that weights split into pairs.

\begin{theo} \label{match}
Let the circle group $S^1$ act on a compact oriented manifold $M$ with a discrete fixed point set. Let $w$ be a positive integer. Then the following holds.
\begin{center}
$\displaystyle \sum_{p \in M^{S^1}} N_p(w) \equiv 0 \mod 2$.
\end{center}
Here, $N_p(w)$ denotes the number of times $w$ occurs as a weight at $p$.
Consequently, we can split the multiset $\{w_{p,i}\}_{1 \leq i \leq n, p \in M^{S^1}}$ of weights over all fixed points, counted with multiplicity, into pairs $(w_{p,i},w_{q,j})$ such that
\begin{enumerate}
\item $w_{p,i}=w_{q,j}$, and 
\item $p \neq q$ or $i \neq j$.
\end{enumerate}
Here, $\dim M=2n$.
\end{theo}

\begin{proof}
By Theorem \ref{t22}, the signature of $M$ satisfies
\begin{center}
$\displaystyle \textrm{sign}(M) = \sum_{p \in M^{S^1}} \epsilon(p) \prod_{i=1}^{n} \frac{1+t^{w_{p,i}}}{1-t^{w_{p,i}}} = \sum_{p \in M^{S^1}} \epsilon(p) \prod_{i=1}^{n} \frac{1+t^{w_{p,i}}}{1-t^{w_{p,i}}}=\sum_{p \in M^{S^1}} \epsilon(p) \prod_{i=1}^{n} [(1+t^{w_{p,i}})\sum_{j=0}^{\infty} t^{jw_{p,i}}]=\sum_{p \in M^{S^1}} \epsilon(p) \prod_{i=1}^{n} (1+2\sum_{j=1}^{\infty} t^{jw_{p,i}})$
\end{center}
for all indeterminates $t$, and is a constant. The idea of the proof  is to compare the coefficients of the $t^w$-terms in the last expression.

As discussed above, for a fixed point $p$, the coefficient of the $t^w$-term in $\displaystyle \epsilon(p) \prod_{i=1}^{n} (1+2\sum_{j=1}^{\infty} t^{jw_{p,i}})$ is
\begin{center}
$\displaystyle \epsilon(p) \sum_{\substack{j_1w_{p,1}+\cdots+j_nw_{p,n}=w \\ 0\leq  j_i}} 2^{m(j_1,\cdots,j_n)}$,
\end{center}
where $j_1,\cdots,j_n$ are non-negative integers and $m(j_1,\cdots,j_n)$ is the number of positive elements in the multiset $\{j_1,\cdots,j_n\}$.
Since the sum of the coefficients of the $t^w$-terms over all fixed points is zero, comparing the coefficients of the $t^w$-terms over all fixed points, it follows that
\begin{center}
$\displaystyle 0=\sum_{p \in M^{S^1}} \epsilon(p) \sum_{\substack{j_1w_{p,1}+\cdots+j_nw_{p,n}=w \\ 0\leq  j_i}} 2^{m(j_1,\cdots,j_n)}$.
\end{center}
For a fixed point $p$, if $m(j_1,\cdots,j_n)$ is greater than $1$ where $j_1w_{p,1}+\cdots+j_nw_{p,n}=w$, the number $2^{m(j_1,\cdots,j_n)}$ is a multiple of $4$. Therefore, modulo $4$, from the above equation it follows that 
\begin{center}
$\displaystyle \sum_{p \in M^{S^1}} \epsilon(p) \sum_{\substack{j_1w_{p,1}+\cdots+j_nw_{p,n}=w \\ 0\leq  j_i \\ m(j_1,\cdots,j_n)=1}}2 \equiv 0 \mod 4$.
\end{center}
Therefore,
\begin{equation} \label{eq3}
\displaystyle \sum_{p \in M^{S^1}} \epsilon(p) \sum_{\substack{j_1w_{p,1}+\cdots+j_nw_{p,n}=w \\ 0\leq  j_i \\ m(j_1,\cdots,j_n)=1}}1 \equiv 0 \mod 2.
\end{equation}

With this, we shall prove this theorem by induction on $w$. First, suppose that there does not exist a weight $w_{p,i}$ that is smaller than $w$ and divides $w$.
Then the cardinality of the set 
\begin{center}
$\{(j_1,\cdots,j_n) \in \mathbb{Z}^n : j_1w_{p,1}+\cdots+j_nw_{p,n}=w , 0\leq  j_i, m(j_1,\cdots,j_n)=1\}$
\end{center}
is equal to $N_p(w)$.
Therefore, the sum
\begin{center} 
$\displaystyle \sum_{\substack{j_1w_{p,1}+\cdots+j_nw_{p,n}=w \\ 0\leq  j_i \\ m(j_1,\cdots,j_n)=1}} 1$
\end{center}
is equal to $N_p(w)$ for each $p \in M^{S^1}$. Hence, by Equation~\eqref{eq3}, this theorem holds.

Next, suppose that this theorem holds for each positive integer $w'$ that is smaller than $w$ and divides $w$. 
Let $k$ be a positive integer. 
Suppose that an element $(j_1,\cdots,j_n)$ of $\mathbb{Z}^n$ satisfies $j_1w_{p,1}+\cdots+j_n w_{p,n}=w$, $j_i \geq 0$ for all $i$, $m(j_1,\cdots,j_n)=1$, and $\max\{j_i\}=k$. That is, exactly one of $j_i$ is $k$, say $j_\ell=k$, all the other $j_i$ are zero, and $j_\ell w_{p,\ell}=w$. Let $j_i'=j_i/k$ for all $i$. Then
$j_1 w_{p,1} + \cdots j_n w_{p,n}=w$ is equivalent to $j_1' w_{p,1} + \cdots j_1' w_{p,n}=w/k$. Moreover, the cardinality of the set
$$\left\{(j_1',\cdots,j_n') \in \mathbb{Z}^n \biggm| \begin{array}{l} j_1'w_{p,1}+\cdots+j_n'w_{p,n}=w/k,0 \leq j_i', \\ m(j_1',\cdots,j_n')=1,\max\{j_i'\}=1 \end{array}\right\}$$
is equal to $N_p(w/k)$ for each positive integer $k$.
Therefore,
\begin{center}
$\displaystyle \sum_{p \in M^{S^1}} \epsilon(p) \sum_{\substack{j_1w_{p,1}+\cdots+j_nw_{p,n}=w \\ 0\leq  j_i \\ m(j_1,\cdots,j_n)=1}} 1$ 

$=\displaystyle \sum_{k} \sum_{p \in M^{S^1}} \epsilon(p) \sum_{\substack{j_1w_{p,1}+\cdots+j_nw_{p,n}=w \\ 0\leq  j_i \\ m(j_1,\cdots,j_n)=1 \\ \max\{j_i\} =k}} 1$

$=\displaystyle \sum_{p \in M^{S^1}} \epsilon(p) \sum_{\substack{j_1w_{p,1}+\cdots+j_nw_{p,n}=w \\ 0\leq  j_i \\ m(j_1,\cdots,j_n)=1 \\ \max\{j_i\} =1}} 1 + \sum_{k>1} \sum_{p \in M^{S^1}} \epsilon(p) \sum_{\substack{j_1w_{p,1}+\cdots+j_nw_{p,n}=w \\ 0\leq  j_i \\ m(j_1,\cdots,j_n)=1 \\ \max\{j_i\} =k}} 1$

$=\displaystyle \sum_{p \in M^{S^1}} \epsilon(p) \cdot N_p(w) + \sum_{k>1} \sum_{p \in M^{S^1}} \epsilon(p) \sum_{\substack{j_1'w_{p,1}+\cdots+j_n'w_{p,n}=w/k \\ 0\leq  j_i' \\ m(j_1',\cdots,j_n')=1 \\ \max\{j_i'\} =1}} 1$

$=\displaystyle \sum_{p \in M^{S^1}} \epsilon(p) \cdot N_p(w) + \sum_{k>1} \sum_{p \in M^{S^1}} \epsilon(p) \cdot N_p(w/k)$.
\end{center}
Since this sum is congruent to $0$ modulo $2$ by Equation~\eqref{eq3} and $$\sum_{p \in M^{S^1}} \epsilon(p) \cdot N_p(w/k) \equiv 0 \mod 2$$ for all $k>1$ by inductive hypothesis, this theorem holds.
\end{proof}

If $w$ is odd, a stronger result holds. To show this, we need the following property for the smallest weight, which follows by comparing the coefficients of $t^w$-terms in Equation~\eqref{eq2} where $w$ is the smallest weight.

\begin{lem} \label{small} \cite{J3, M} Let the circle act on a $2n$-dimensional compact oriented manifold $M$ with a discrete fixed point set. Let $w=\min\{w_{p,i} \, | \, 1 \leq i \leq n, p \in M^{S^1}\}$. Then
\begin{center}
$\displaystyle \sum_{p \in M^{S^1}, \epsilon(p)=+1} N_p(w)=\sum_{p \in M^{S^1}, \epsilon(p)=-1} N_p(w)$,
\end{center}
where $N_p(w)=|\{i \, : \, w_{p,i}=w, 1 \leq i \leq n, p \in M^{S^1}\}|$ is the number of times $w$ occurs as a weight at $p$. \end{lem}

Given an $S^1$-manifold $M$, the cyclic group $\mathbb{Z}_w$ of order $w$ also acts on $M$ as a subgroup of $S^1$.
For $w$ odd, the $\mathbb{Z}_w$-fixed set of an orientable $S^1$-manifold is orientable.

\begin{lem} \label{odd-ori}
Let the circle group act on an orientable manifold $M$. Let $w \geq 1$ be an odd integer. Then any component of the $\mathbb{Z}_w$-fixed set $M^{\mathbb{Z}_w}$ is orientable.
\end{lem}

\begin{proof}
Let $F$ be a component of $M^{\mathbb{Z}_w}$. Let $m$ be a point in $F$. Since $w$ is odd, the $\mathbb{Z}_w$-representation of the normal space $N_mF$ of $F$ in $M$ at $m$ decomposes into irreducibles, where each irreducible is isomoprhic to a complex one-dimensional vector space on which a generator of $\mathbb{Z}_w$ acts as multiplication by a root of unity. Thus, the $\mathbb{Z}_w$-action on each irreducible is orientation preserving and so is $N_mF$. This holds for all $m \in F$ and hence the normal bundle $NF$ of $F$ in $M$ is orientable, and thus $F$ is orientable.
\end{proof}

We show that odd weights can be split into pairs, such that weights in each pair come from distinct fixed points.

\begin{theo} \label{odd}
Let the circle group $S^1$ act effectively on a compact oriented manifold $M$ with a discrete fixed point set. Let $w$ be an odd positive integer. Then the following hold.
\begin{enumerate}
\item For any component $F$ of $M^{\mathbb{Z}_w}$,
\begin{center}
$\displaystyle \sum_{p \in F,\epsilon_F(p)=+1} N_p(w)=\sum_{p \in F,\epsilon_F(p)=-1} N_p(w)$,
\end{center}
where $F$ is given any orientation. 
\item We can split the multiset $\{w_{p,i} \mid w_{p,i}=w, 1 \leq i \leq n, p \in M^{S^1}\}$ of weights that are equal to $w$ into pairs $(w_{p,i},w_{q,j})$ so that $p$ and $q$ are distinct, and $p$ and $q$ lie in the same component of $M^{\mathbb{Z}_{w}}$.
\item For any fixed point $q$, 
\begin{center} 
$\displaystyle N_q(w) \leq \sum_{p \in M^{S^1}, p \neq q} N_p(w)$.
\end{center}
\end{enumerate}
Here, $N_p(w)=|\{i \, : \, w_{p,i}=w, 1 \leq i \leq n, p \in M^{S^1}\}|$ is the number of times $w$ occurs as a weight at $p$.
\end{theo}

\begin{proof}
Let $F$ be a component of $M^{\mathbb{Z}_w}$. The component $F$ is a closed submanifold of $M$. By Lemma~\ref{odd-ori}, $F$ is orientable; pick any orientation of $F$. The circle group $S^1$ acts on $F$ as a restriction, and its fixed point set $F^{S^1}$ is $F^{S^1}=M^{S^1} \cap F$ and thus discrete. Moreover, $x$ is a weight in $T_pF$ if and only if $x$ is a multiple of $w$. This means that $N_p^F(w)=N_p(w)$, where $N_p^F(w)$ is the multiplicity of the weight $w$ in $T_pF$.

We apply Lemma~\ref{small} to the $S^1$-action on $F$ to have
\begin{center}
$\displaystyle \sum_{p \in F^{S^1}, \epsilon_F(p)=+1} N_p(w)=\sum_{p \in F^{S^1}, \epsilon_F(p)=-1} N_p(w)$.
\end{center}
Thus, the first claim holds. Moreover, the displayed equation means that each time $w$ occurs as a weight $w_{p,i}$ at a fixed point $p$ with $\epsilon_F(p)=+1$, 
$w$ occurs as a weight $w_{q,j}$ at a fixed point $q$ with $\epsilon_F(p)=+1$; we take $(w_{p,i},w_{q,j})$ as a pair. Repeating this argument for all fixed components of $M^{\mathbb{Z}_w}$, the second claim follows. 
The last claim follows from the second claim.
\end{proof}

\end{document}